\documentclass[a4paper, 11pt, leqno]{amsart}

\usepackage[dvipsnames]{xcolor}

  \usepackage[utf8]{inputenc}
  \usepackage[T1]{fontenc}
\usepackage{lmodern}

\usepackage{geometry}
\usepackage[textwidth=28mm]{todonotes}
\usepackage{mathtools}
\usepackage{yfonts}
\usepackage[foot]{amsaddr}

\usepackage[pdfencoding=auto]{hyperref}
\hypersetup{
    colorlinks,
    citecolor=green,
    filecolor=black,
    linkcolor=blue,
    urlcolor=black
}

\usepackage[english]{babel}

\usepackage{dsfont}
\usepackage{soul}
\usepackage{comment}
\usepackage{cases}
\usepackage{enumerate}
\usepackage[shortlabels]{enumitem}
\usepackage{tikz}
\usepackage{pgfplots}
\pgfplotsset{compat=1.15}
\usetikzlibrary{arrows}
\usetikzlibrary{patterns}

\numberwithin{equation}{section}

\usepackage{tabularx}

\newcommand{\nocontentsline}[3]{}
\let\origcontentsline\addcontentsline
\newcommand\stoptoc{\let\addcontentsline\nocontentsline}
\newcommand\resumetoc{\let\addcontentsline\origcontentsline}

\theoremstyle{plain}
\newtheorem{thm}{Theorem}[section]

\newtheorem{rem}[thm]{Remark} 

\newtheorem{prop}[thm]{Proposition} 
\theoremstyle{definition} 
\newtheorem{defi}{Definition}[section] 
\theoremstyle{remark}

\newcommand{\R}{\mathbb{R}}

\newcommand{\x}{x}

\newcommand{\uv}{v}
\newcommand{\f}{\mathbf{f}}

\newcommand{\B}{\mathbf{B}}

\newcommand{\eps}{\varepsilon}

\makeatletter
\newcommand{\leqnomode}{\tagsleft@true}
\newcommand{\reqnomode}{\tagsleft@false}
\makeatother

\author{Gabriella Puppo}
\author{Thomas Rey}
\author{Tommaso Tenna}
\address[Gabriella Puppo]{Dipartimento di Matematica, Sapienza Università di Roma, P.le Aldo Moro 5, 00185 Rome, Italy}
\address[Thomas Rey]{Université Côte d’Azur, CNRS, LJAD, Parc Valrose, F-06108 Nice, France}
\address[Tommaso Tenna]{Université Côte d’Azur, CNRS, LJAD, Parc Valrose, F-06108 Nice, France -- Dipartimento di Matematica, Sapienza Università di Roma, P.le Aldo Moro 5, 00185 Rome, Italy}
\title[Kinetic derivation of an isentropic two-phase flow model]{Formal derivation of an isentropic two-phase flow model from the multi-species Boltzmann equation}
\date{}

\begin{document}

\begin{abstract}
    Starting from the multi-species Boltzmann equation for a gas mixture, we propose the formal derivation of the isentropic two-phase flow model introduced in \cite{romenski2004compressible}.  We examine the asymptotic limit as the Knudsen numbers approach zero, in a regime characterized by resonant intra-species collisions, where interactions between particles of the same species dominate. This specific regime leads to a multi-velocity and multi-pressure hydrodynamic model, enabling the explicit computation of the coefficients for the two-phase macroscopic model. Our derivation also accounts for the inclusion of the evolution of the volume fraction, which is a key variable in many macroscopic multiphase models.\\
    \medskip
    
\end{abstract}


\maketitle

\vspace{-1cm}

\section{Introduction}
Multiphase flow models are widely diffused for the description of fluids with different physical or thermodynamic properties, usually separated by interfaces which are assumed to be diffuse. In these diffuse interface models, the fluids have different macroscopic properties, but are allowed to mix across the interface zone. One of the first models of this type was introduced by Baer and Nunziato in \cite{baer1986} and it is still the basis for the development of new strategies to describe compressible two-phase flows. After this pioneering work, many other models were proposed to approximate multi-phase fluids, introducing some slight modifications to the original system regarding closure relations and relaxation parameters, see \cite{dumbser2011, flatten2011, lund2012, saurel1999, saurel2001}. Their derivation is mainly based on phenomenological assumptions, which affect the validity of the model in more general situations. An alternative approach to the Baer-Nunziato model is the hyperbolic multi-phase model introduced by Romenski, Drikakis and Toro in \cite{toro2010, romenski2004compressible}, based on the theory of symmetric hyperbolic thermodynamically compatible systems \cite{godunov1995}. 
Despite the fact that this model has been validated from a numerical point of view (see \cite{puppo2023} for its approximation in all Mach number regimes and \cite{theindumbser2022} for a complete analysis on the Riemann problem in the barotropic case), the derivation has an intrinsic phenomenological nature and it contains some coefficients which cannot be explicitly determined. The purpose of the present paper is a formal derivation of macroscopic equations for multiphase flows from Boltzmann equations for mixtures of rarefied gases.

The behavior of rarefied gases is usually modeled through kinetic theory, in order to reproduce non-equilibrium phenomena. The seminal Boltzmann Equation represents the cornerstone to describe the dynamics of a gas which is not in thermodynamical equilibrium, composed by a very large number of identical particles \cite{cercignani1988, cercignanipulvirenti}.

The necessity of studying systems involving gases composed of molecules of different types leads to kinetic models for mixtures. Indeed, realistic kinetic approximations of such situations can be performed only through multi-species models, in which interactions and exchanges between different components must be properly described.
Whereas the mathematical properties of the Boltzmann equation for a single-species gas are well known, the analytical study of multi-species kinetic models is still not complete. Most of the literature concerns theory for Bhatnagar-Gross-Krook (BGK) models for gas mixtures \cite{andries2002, bobylev2018, boscarino2021, haack2017, klingeberg2019}. To the best of our knowledge, the first derivation of the Boltzmann equation for binary gas mixtures is due to Chapman and Cowling \cite{CC_1939}. More recently, Briant and coauthors \cite{briant2016stability, briant2016boltzmann} have investigated the existence, uniqueness, positivity and exponential trend to equilibrium for the linearized version of the full non-linear Boltzmann equation in a perturbative isotropic $L_v^{1}L_x^{\infty}$ setting, namely $L^1$ in velocity space and $L^\infty$ in the physical space. A spectral gap for the linearized operator around global equilibrium has been proven in \cite{briant2016boltzmann}, then extended in \cite{bondesan2020} to the case of the linearized operator around non-equilibrium states and in \cite{bondesan2024}, providing the explicit spectral gap for the linearized operator modeling reactive mixtures.
In another recent paper, Gamba and Pavi{\'c}-{\v{C}}oli{\'c} \cite{gambapavic2020} have given an existence and uniqueness result for the spatially homogeneous multi-species Boltzmann equations with binary interactions in the Banach space represented by the $L^1$ space weighted polynomially in velocity. Properties of such solutions have been later analyzed in \cite{delacanal2020}. Further investigations on the mixing of monatomic and polyatomic gases can be found in \cite{alonsocolicgamba2024}. 

The formal hydrodynamic limit for kinetic mixtures has been investigated for different regimes, both for the BGK model \cite{bisi2021hydro}, for a mixed BGK-Boltzmann model \cite{bisi2024, bisigroppi2025} and for the Boltzmann model \cite{biancadogbe2015, bisi2014, bisi2011multi, bisi2012multi,boudin2017, boudin2015}.
Let us mention also the work of Briant and Grec \cite{briantgrec2023} on the rigorous derivation of the Fick cross-diffusion system from the multi-species Boltzmann equation. 

All hydrodynamic limits of the Boltzmann equation correspond to situations where the frequency of collisions is rapidly increasing. This means that the collision mechanism governed by the collision operator holds on a time scale which is much smaller than the observation time scale and we have an almost instantaneous convergence of the distribution function to thermodynamic equilibrium. We refer to the book of Saint-Raymond for a complete review of hydrodynamic limits of the Boltzmann equation in the single-species framework \cite{saintraymond2009}.
In any case, the fluid limit provides a set of equations for the macroscopic fields, obtained as moments of the kinetic distribution function. This approach leads to a hierarchy of equations coupling lower order moments to those of higher order, which could be truncated to obtain the desired order of approximation \cite{bisi2021hydro}. In a very recent paper \cite{paviccolic2024} the author has derived a multi-velocity and multi-temperature macroscopic model without phase separation obtained within extended thermodynamics from Boltzmann-like equations, for mixtures of monatomic and polyatomic gases. 

The aim of this paper is to propose a formal derivation of the model proposed by Romenski, Drikakis and Toro \cite{romenski2004compressible} using kinetic theory. In particular, after a brief introduction to the multi-species Boltzmann equation, we discuss a hydrodynamic limit for this model in the regime of resonant intra-species collisions \cite{galkin1994}, namely when the interactions between particles of the same species represent the fast process of the dynamics. This regime allows fast convergence towards local ``equilibria'', in which each species reaches its own velocity, whereas inter-species interactions are characterized by a much slower process. The novelty of the work lies in the justification of the different scalings, in the introduction of the notion of ``volume fraction'' starting from the kinetic formulation and in the use of the full Boltzmann collision operator. Moreover, deriving a constitutive law from the microscopic dynamics at the interface of two gases, we are able to recover the original model and to propose explicit expressions for the coefficients appearing in the macroscopic model. 
\section{The Multi-species Boltzmann Equation}
The derivation of the multi-species Boltzmann equation is mostly formal, unlike that of the single-species case. Indeed, it is possible to consider some modeling assumptions, extending the microscopic dynamics of the single-species case to gas mixtures, in which different species are taken into account.

Let us consider a mixture of $M$ species, each described by a distribution function $f_p = f_p(t,x,v)$, where the mass of a molecule of component $p$ is denoted by $m_p$. Given $\x \in \Omega \in \R^{d}$, $\uv \in \R^{d}$, the distribution functions $f_p$ evolve according to the so-called multi-species Boltzmann equation, which, in absence of external forces, takes the following form \cite{briant2016boltzmann}
\begin{equation}
	\label{BEs_system_mixtures}
	\begin{cases}
		\partial_t f_p + \uv \cdot \nabla_\x f_p = \displaystyle \frac{1}{\eps_p} \mathcal{Q}_p(\f), \qquad p=1,\dots,M,\\
		\mathcal{Q}_p (\f) = \displaystyle \sum_{q=1}^M Q^{pq} (f_p,f_q),\\ 
		f_p(0,\x,\uv)= f_{p,0} (\x,\uv),
	\end{cases}
\end{equation}
where $\f=(f_1,\dots, f_M)$. The parameter $\eps_p >0$ is the \textit{Knudsen number}, that is the ratio between the mean free path of particles before each collision and the length scale of observation. Therefore, $\eps_p$ governs the frequency of collisions for the species $p$: taking a small Knudsen number corresponds to more frequent collisions. This particular scaling of the Boltzmann equation is usually known as \textit{hyperbolic scaling}.\\
The time evolution of the distribution function due to particle collisions in the Boltzmann equation is given by a sum of \textit{collision operators} $Q^{pq}$, referred to as self-collision operators (or \textit{intra-species} collision operators) if $p=q$ and cross-collision operators (or \textit{inter-species} collision operators) if $p \neq q$.
The microscopic dynamics is fundamental to understand the formulation of the Boltzmann equation and to derive an explicit formulation for the collision operator in \eqref{BEs_system_mixtures}. We need the following reasonable assumptions:
\begin{enumerate}[label=(\roman*)]
	\item \label{microscopic_hypotheses:one} The particles interact via \textit{binary} collisions. This means that collisions between $3$ or more particles can be neglected.
	\item \label{microscopic_hypotheses:two} These binary collisions are localized in space and time.
	\item \label{microscopic_hypotheses:three} Collisions preserve mass for each species, total momentum and total kinetic energy, namely for two particles of species $p$ and $q$ involved in the collision it holds 
	\begin{equation*}
		\begin{cases}
			m_p v' + m_q v_*' = m_p v + m_q v_*,\\
			\displaystyle m_p |v'|^2 + m_q |v_*'|^2 = m_p |v|^2 + m_q |v_*|^2.
		\end{cases}
	\end{equation*}
\end{enumerate}
We avoid any superscripts $v^p$ in the notations and we simply refer to the pre-collisional velocities as $(v, v_*)$ and to the post-collisional ones as $(v', v_*')$.\\
In particular, the conservation of total momentum and total energy yields the following parametrizations for the post-collisional velocities $(v', v_*')$ in terms of the pre-collisional ones $(v,v_*)$:
\begin{itemize}
	\item The $\omega$-representation
	\begin{equation}
		\label{omega_representation}
		\begin{aligned}
			v'=v - \frac{2\,m_q}{m_p+m_q} ((v-v_*) \cdot \omega)\, \omega,\\
			\displaystyle v_*'=v_* + \frac{2\,m_p}{m_p+m_q} ((v-v_*) \cdot \omega)\, \omega,\\
		\end{aligned} 
	\end{equation}
	where $\omega \in \mathbb{S}^{d-1}$.
	\item The $\sigma$-representation 
	\begin{equation}
		\label{sigma_representation}
		\begin{aligned}
			v'=\frac{m_p v+ m_q v_*}{m_p+m_q} + \frac{m_q}{m_p+m_q} |v-v_*| \sigma,\\
			v_*'=\frac{m_p v+ m_q v_*}{m_p+m_q} - \frac{m_p}{m_p+m_q} |v-v_*| \sigma,\\
		\end{aligned} 
	\end{equation}
	where $\sigma \in \mathbb{S}^{d-1}$.
\end{itemize}
	
	The $\sigma$-representation is related to the $\omega$-one through the following relation
	\begin{equation}
		((v-v_*) \cdot \omega)\, \omega = \frac{1}{2} ( v-v_*- |v-v_*| \sigma).
	\end{equation}
Defining $\mathcal{V}$ as the velocity of the center of mass of the particles $p$ and $q$:
\begin{equation}
	\label{center_mass}
	\mathcal{V}=\frac{m_p}{m_p+m_q}v+\frac{m_q}{m_p+m_q}v_*,
\end{equation}
the laws of elastic collisions gives 
\begin{equation}
	v=\mathcal{V}+\frac{m_q}{m_p+m_q}(v-v_*) \qquad v_*=\mathcal{V}-\frac{m_p}{m_p+m_q}(v-v_*).
\end{equation}
and we can also rewrite the post-collisional velocities as 
\begin{equation}
	v'=\mathcal{V}+\frac{m_q}{m_p+m_q}(v'-v_*) \qquad v'_*=\mathcal{V}-\frac{m_p}{m_p+m_q}(v'-v'_*).
\end{equation}
Hence, the points $v$, $v_*$, $v'$ and $v_*'$ belong to the plane defined by $\mathcal{V}$ and $\text{Span} (\sigma, v-v_*)$ and we have the geometric configuration in Fig.\ref{geometrical_configuration}. The angle $\theta$ depends on $\sigma$ according to the relation 
\begin{equation}
\cos \theta := \frac{v-v_*}{|v-v_*|}\cdot \sigma.    
\end{equation}
\begin{figure}
	\begin{tikzpicture}[line cap=round,line join=round,>=triangle 45,x=1cm,y=1cm, scale=0.7]
		\draw [shift={(-1,0)},line width=1pt] (0,0) -- (180:0.7431493232436724) arc (180:225:0.7431493232436724) -- cycle;
		\draw [line width=1pt] (-1,0) circle (4cm);
		\draw [line width=1pt] (-1,0) circle (6cm);
		\draw [line width=1pt] (-3.8284271247461903,-2.8284271247461903)-- (3.2426406871192857,4.242640687119286);
		\draw [line width=1pt] (-5,0)-- (5,0);
		\draw [line width=1pt] (-1,0)-- (3.2426406871192857,4.242640687119286);
		\draw [line width=1pt] (1.016223158379683,2.226417528739602) -- (1.2264175287396024,2.016223158379684);
		\draw [line width=1pt] (-1,0)-- (5,0);
		\draw [line width=1pt] (2,0.14862986464873434) -- (2,-0.14862986464873434);
		\draw [line width=1pt] (-5,0)-- (-1,0);
		\draw [line width=1pt] (-3.059451945859494,0.14862986464873434) -- (-3.059451945859494,-0.14862986464873434);
		\draw [line width=1pt] (-2.9405480541405065,0.14862986464873434) -- (-2.9405480541405065,-0.14862986464873434);
		\draw [line width=1pt] (-3.8284271247461903,-2.8284271247461903)-- (-1,0);
		\draw [line width=1pt] (-2.5613496216250384,-1.3511552512651215) -- (-2.351155251265119,-1.5613496216250373);
		\draw [line width=1pt] (-2.4772718734810706,-1.2670775031211539) -- (-2.267077503121151,-1.4772718734810697);
		\draw [->,line width=0.6pt,color=blue] (-1,0) -- (-1.9218160309167933,-0.9218160309167933);
		\begin{scriptsize}
			\draw [fill=black] (5,0) circle (2pt);
			\draw[color=black] (5.576578367574754,0.2) node {\Large $v_*$};
			\draw [fill=black] (-5,0) circle (2pt);
			\draw[color=black] (-5.5,0.2) node {\Large $v$};
			\draw [fill=black] (-1,0) circle (2pt);
			\draw[color=black] (-1.109434604075386,0.6508002735590815) node {\Large $\mathcal{V}$};
			\draw [fill=black] (3.2426406871192857,4.242640687119286) circle (2pt);
			\draw[color=black] (3.6119238835089276,4.69148158880605) node {\Large $v_*'$};
			\draw [fill=black] (-3.8284271247461903,-2.8284271247461903) circle (2pt);
			\draw[color=black] (-4.25,-3.2561814525396556) node {\Large $v'$};
			\draw[color=black] (-2.1195733250562734,-0.37125294140357772) node {\Large $\theta$};
			\draw[color=blue] (-1.2658759476720945,-1.0252243458580087) node {\Large $\mathbf{\sigma}$};
		\end{scriptsize}
	\end{tikzpicture}
	\caption{Geometrical configuration of elastic collisions.}
	\label{geometrical_configuration}
\end{figure}
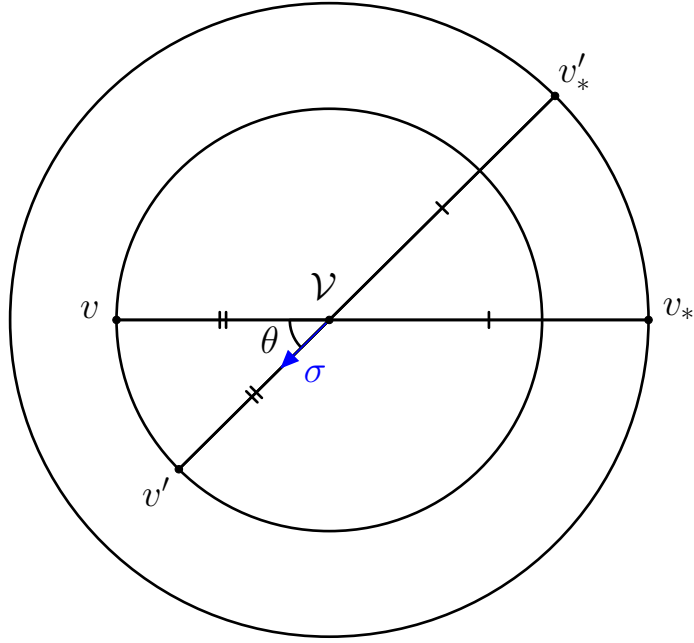

Using the microscopic hypotheses \ref{microscopic_hypotheses:one}-\ref{microscopic_hypotheses:two}-\ref{microscopic_hypotheses:three}, it is possible to derive the explicit expression for the strong form of the collision operators as
\begin{equation}
	\label{strong_form}
	Q^{pq}(f_p,f_q)= \int_{\R^d} \int_{\mathbb{S}^{d-1}} \mathbf{B}_{pq}(|v-v_*|,\theta) \left[f_q(v_*')\,f_p(v') - f_q(v_*)\,f_p(v) \right] d\Omega \, dv_*.
\end{equation}
The collision kernels $\mathbf{B}_{pq}$ are non-negative functions, which contain all the physical information about particles interactions, we refer to 
\cite{cercignani1988, cercignanipulvirenti} for a detailed analysis on this subject.

In the following, we consider a collision kernel $\mathbf{B}_{pq}$ which depends only on $|v-v_*|$ and $\cos \theta$.
For a general spherically symmetrical intermolecular potential $\Phi(r)$, the deflection angle $\theta$ is given by \cite{bird1994, chapmancowling1939}:
\begin{equation}
    \theta(|u|) = \pi - 2\int_0^{W_1} \left[ 1-W^2-\frac{4\Phi(r)}{m\,|u|^2} \right]^{-1/2}\,dW,
\end{equation}
where $W=d_p\cos(\theta/2)/r$, $W_1$ is the positive root of the term in the brackets, $d_p$ is the molecular diameter and $r$ is the distance between two molecules. For the case of $(k-1)$th inverse-power laws, namely for a potential of the form $\Phi(r) \propto 1/(k-1)\,r^{k-1}$, the relevant simplification is that $\mathbf{B}_{pq}(|v-v_*|,\theta)$ becomes the product of a function of $\theta$ alone by a fractional power of $|v-v_*|$. Thus, in the case of inverse-power law potentials, it is possible to rewrite the collision kernel as \cite{bird1994}
\begin{equation}
\label{B_collisionkernel}
	\mathbf{B}_{pq}(|v-v_*|,\theta)= b_{pq}(\theta) |v-v_*|^{\gamma}, \qquad \gamma= \displaystyle \frac{k-(2d-1)}{k-1},
\end{equation}
where $b_{pq}(\theta)$ is a nonelementary function of $\theta$.
In particular, in $d=3$ dimensions of space, we obtain $\gamma=(k-5)/(k-1)$. Thus, the choice $k=5$ corresponds to the particular case in which the collision kernel does not depend on the relative velocity, but only on the deviation angle. Such particles are called Maxwellian molecules and $\mathbf{B}_{pq}(|v-v_*|,\theta)=b_{pq}(\cos \theta)$. The pseudo-Maxwellian case is given by a constant value of $b_\gamma(\cos \theta)$. For numerical purpose a widely diffused model is the variable hard sphere (VHS) model introduced by Bird \cite{bird1994}, corresponding to the choice $b_{pq}(\cos \theta)=C_\gamma$, where $C_\gamma>0$. The case $\gamma=0$ is the aforementioned pseudo-Maxwellian molecules case, whereas $\gamma=1$ gives the hard-spheres molecules case.\\
The hypothesis of constant $\gamma$ for different species could be weakened, for instance by considering a collision kernel \eqref{B_collisionkernel} with different $\gamma_p$ according to the species $p$, as done in \cite{alonsocolicgamba2024}, but it does not affect the main conclusions of this work.

\subsection{Weak Form}
In this section we investigate the weak form of the multi-species Boltzmann equation, following the clear derivation presented in \cite{boudin2015} (see also \cite{cercignani1988} for the single-species case) . The operators $Q^{pp}$ are the classical Boltzmann collision operators. Starting from the definition of velocities obtained in \eqref{sigma_representation} and using the microscopic hypotheses, we derive the \textit{weak form} in the $\sigma$-representation, considering a $C^\infty$ test function $\psi: \R^d \to \R$. We assume that the cross section $\B_{pq}$ satisfies the micro-reversibility assumption $\B_{pq}(v,v_*,\sigma)=\B_{pq}(v_*,v,\sigma)$. We obtain
\begin{equation}
\begin{aligned}
	\label{weak_monospecies}
	\int_{\R^d} Q^{pp} &(f_p,f_p) (v) \psi(v) \, dv \\ = &\frac{1}{2} \int_{\R^d \times \R^d \times \mathbb{S}^{d-1}} f_{p*} f_p (\psi'+\psi'_*-\psi-\psi_*) \B_{pp}(|v-v_*|, \cos \theta) d\sigma \, dv \, dv_*,
\end{aligned}
\end{equation}
where we have introduced the shorthands $f_p'=f_p(v')$, $f_{p*}=f_p(v_*)$, $f'_{p*}=f_p(v'_*)$, $\psi'=\psi(v')$, $\psi_*=\psi(v_*)$ and $\psi'_*=\psi(v_*')$.\\
In an analogous way, it is possible to characterize the bi-species collision operators, under the assumption of a collision kernel independent from the species. Therefore, considering the definition \eqref{sigma_representation} with $p \neq q$, in the same way of the mono-species framework, we obtain the following formulation of $Q^{pq}$ in the $\sigma$-representation
\begin{multline}
	\label{weak_bispecies}
	\int_{\R^d} Q^{pq} (f_p,f_q) (v) \psi(v) \, dv\\ = \int_{\R^d \times \R^d \times \mathbb{S}^{d-1}} f_p f_{q*}(\psi'-\psi) \B_{pq}(|v-v_*|, \cos \theta) d\sigma \, dv \, dv_*.
\end{multline}
In addition, if we consider both the bi-species collision operators, we can give the following characterization:
\begin{multline}
	\label{weak_bispecies_sum}
		\int_{\R^d} Q^{pq} (f_p,f_q) (v) \psi(v) \, dv + \int_{\R^d} Q^{qp} (f_q,f_p) (v) \phi(v) \, dv\\ = \int_{\R^d \times \R^d \times \mathbb{S}^{d-1}} f_p f_{q*} (\psi'+\phi'_*-\psi-\phi_*) \B_{pq}(|v-v_*|, \cos \theta) d\sigma \, dv \, dv_*,
\end{multline}
for any $C^\infty$ test functions $\psi, \, \phi: \R^d \to \R$.

Using the weak form, one can easily observe that the conservation properties of the Boltzmann equation are extended to the multi-species cases. The mass of each species is conserved, since
\begin{equation}
	\int_{R^d} \mathcal{Q}_p (f, g) (\uv) d\uv = 0,
\end{equation}
for any functions $f(\uv)$ and $g(\uv)$, namely the constant function $v \mapsto 1$ is a collisional invariant for this operator.
On the other hand, the functions $\uv$ and $|\uv|^2$ are not collision invariants, but it holds
\begin{equation}
	\int_{R^{d}} \uv Q^{pq} (f, g) (\uv) d\uv = - \int_{R^{d}} \uv Q^{qp} (g, f) (\uv) d\uv,
\end{equation}
and
\begin{equation}
	\int_{R^{d}} |\uv|^2 Q^{pq} (f, g) (\uv) d\uv = - \int_{R^{d}} |\uv|^2 Q^{qp} (g, f) (\uv) d\uv.
\end{equation}
In particular, the choice of $\psi(v)=m_i\, v$ and $\phi(v)=m_j \, v$ in \eqref{weak_bispecies_sum} leads to the conservation of the total momentum, whereas $\psi(v)=m_i\, |v|^2$ and $\phi(v)=m_j \, |v|^2$ leads to the conservation of the total kinetic energy of the mixture.\\

\subsection{Entropy dissipation}
The large time behavior of the multi-species Boltzmann equation \eqref{BEs_system_mixtures} is here briefly described. In this subsection we restrict our analysis to the case of a bi-species mixture for the sake of simplicity, but it can be easily extended to the case of $M$-species mixtures. In particular, we will consider the space homogeneous setting 
\begin{equation}
	\label{BEs_Homo}
	\begin{cases}
		\partial_t f_1 = Q^{11} (f_1, f_1) + Q^{12}(f_1, f_2),\\
		\partial_t f_2 = Q^{21} (f_2, f_1) + Q^{22}(f_2, f_2).
	\end{cases}
\end{equation} 
\begin{defi}[Multi-species Entropy]
	Let $(f_1, f_2)$ be a solution to the homogeneous Boltzmann equation \eqref{BEs_Homo}. The multi-species entropy is described by the so-called Boltzmann's $H$-functional
	\begin{equation}
		H=\int_{\R^d} f_1 \log f_1(v) \, dv + \int_{\R^d} f_2 \log f_2(v) \, dv.
	\end{equation}
\end{defi}
\begin{rem}
This definition can be easily extended to the case of $M \geq 2$ species by defining
\begin{equation}
	H = \sum_{i=1}^M \int_{\R^d} f_p \, \log f_p (v) \, dv.
\end{equation}
\end{rem}
Let us now recall the following result on entropy dissipation, proven in \cite{desvillettes2005}, here explicitly detailed for a non-reactive mixture of two species.
\begin{prop}
\label{prop_entropy}
	Let us assume strictly positive collision kernels $\B_{pq}$ and let us consider a positive solution $(f_1, f_2)$ to \eqref{BEs_system_mixtures}. Then, the entropy is dissipated, namely
	\begin{equation}
		\frac{dH}{dt} \leq 0.
	\end{equation}
	Moreover, we have $\displaystyle \frac{dH}{dt} = 0$ if and only if the equilibrium $(f_1, f_2)$ is given by a pair of global Maxwellian distributions $(\mathcal{M}_1,\mathcal{M}_2)$ with common average velocity $\bar{\uv}$ and temperature $T_{eq}$:
	\begin{equation}
		\label{global_maxwellian}
		\mathcal{M}_p^{n_p,\bar{\uv},T}(t,\x,\uv)=n_p \Big(\frac{m_p}{2 \pi T_{eq}} \Big)^{\frac{d}{2}} \, \exp\Big(-\frac{m_p}{2T_{eq}} |\bar{\uv}-\uv|^2 \Big), \qquad p=1,2.
	\end{equation}
\end{prop}

\begin{proof}[Sketch of the proof]
Let us consider the expression
	\begin{equation}
		\begin{aligned}
		\frac{dH}{dt} = &\int_{\R^d} \mathcal{Q}^{11}(f_1,f_1) \log f_1(v) \, dv + \int_{\R^d} \mathcal{Q}^{12}(f_1,f_2) \log f_1(v) \, dv +\\
		+ &\int_{\R^d} \mathcal{Q}^{21}(f_2,f_1) \log f_2(v) \, dv + \int_{\R^d} \mathcal{Q}^{22}(f_2,f_2) \log f_2(v) \, dv.
		\end{aligned}
	\end{equation}
	We consider the weak formulation of the collision operators \eqref{weak_monospecies}-\eqref{weak_bispecies_sum} with $\phi=\log f_1$ and $\psi=\log f_2$:
	\begingroup
	\allowdisplaybreaks
	\begin{align*}
	\frac{dH}{dt}&= \frac{1}{2} \int_{\R^d \times \R^d \times \mathbb{S}^{d-1}} f_{1*} f_1 \log \left(\frac{f_1' f_{1*}'}{f_1 f_{1*}} \right) \B_{11} d\sigma \, dv \, dv_* \\
	&-\frac{1}{2} \int_{\R^d \times \R^d \times \mathbb{S}^{d-1}} [f_1' f_{2*}' - f_1 f_{2*}] \log \left(\frac{f_1' f_{2*}'}{f_1 f_{2*}} \right) \B_{12} d\sigma \, dv \, dv_*\\
	&+\frac{1}{2} \int_{\R^d \times \R^d \times \mathbb{S}^{d-1}} f_{2*} f_2 \log \left(\frac{f_2' f_{2*}'}{f_2 f_{2*}} \right) \B_{22} d\sigma \, dv \, dv_*\\
	&=\frac{1}{2} \LARGE \sum_{p,q \in \{1,2\}} \int_{\R^d \times \R^d \times \mathbb{S}^{d-1}} f_{p*} f_q \left[\log \left(\frac{f_p' f_{q*}'}{f_p f_{q*}} \right) - \frac{f_p' f_{q*}'}{f_p f_{q*}} + 1 \right] \B_{pq} d\sigma \, dv \, dv_*\\
	&+\frac{1}{2} \displaystyle \sum_{p,q \in \{1,2\}} \int_{\R^d \times \R^d \times \mathbb{S}^{d-1}} (f_{q*}' f_p' - f_{q*} f_p) \B_{pq} d\sigma \, dv \, dv_* \Big] \leq 0.
	\end{align*}
	\endgroup
	Indeed, since we are assuming strictly positive cross sections, $\log x - x +1 \leq 0$ and the collisional transformation $(v,v_*) \to (v',v'_*)$ is an involution, all terms on the right hand side of the equation above are non-positive and this proves the entropy dissipation.
	
	To prove the second part, we first observe that, since we are considering a sum of non-positive contributions, $\displaystyle \frac{dH}{dt}=0$ implies that all these terms must be zero. The terms corresponding to intra-species collisions (namely depending only on $f_p$) vanish if there exist $n_p \geq 0$, $\bar{\uv}_p \in \R^d$ and $T_p > 0$ such that
	\begin{equation}
		\label{local_Maxwellian}
		f_p(v) = n_p \Big(\frac{m_p}{2 \pi T_p} \Big)^{\frac{d}{2}} \, \exp\Big(-\frac{m_p}{2T_p} |\bar{\uv}_p-\uv|^2 \Big), \qquad p=1,2.
	\end{equation}
	To conclude the proof, in \cite{desvillettes2005} the authors prove that $v_1=v_2=\bar{v}$ and $T_1=T_2=T_{eq}$.
\end{proof}

\begin{rem}
Let us underline that the \textit{local} Maxwellians $\mathcal{M}_p(t,\x,\uv)$ defined in \eqref{local_Maxwellian} are maximum of the Boltzmann entropy in the mono-species setting, but they are not local equilibria for the gas mixture. Nevertheless, they play a fundamental role in the derivation of hydrodynamic limits for different regimes.
\end{rem}

Let us consider the inhomogeneous system with a hyperbolic scaling $\eps$, under the assumption that the relaxation velocity is the same for each collision operator, namely
\begin{equation}
	\begin{cases}
		\partial_t f_1 + v \cdot \nabla_x f_1= \displaystyle \frac{1}{\eps}Q^{11} (f_1, f_1) + \frac{1}{\eps}Q^{12}(f_1, f_2),\\[10pt]
		\partial_t f_2 + v \cdot \nabla_x f_2= \displaystyle \frac{1}{\eps} Q^{21} (f_2, f_1) +\frac{1}{\eps} Q^{22}(f_2, f_2).
	\end{cases}
\end{equation} 
In this framework, the equilibrium $(f_1, f_2)$ is reached in the limit $\eps \to 0$ and the distribution functions converge towards Maxwellian functions with common average velocity and temperature, as shown in Proposition \ref{prop_entropy}.

In mixtures of gases with disparate masses the exchange of momentum between particles of the same species is usually a faster process and the relaxation velocity is not the same for each collision operator. In this case, the inhomogeneous system is scaled as follows
\begin{equation}
\label{system_2knudsen}
	\begin{cases}
		\partial_t f_1 + v \cdot \nabla_x f_1= \displaystyle \frac{1}{\eps}Q^{11} (f_1, f_1) + \frac{1}{\kappa}Q^{12}(f_1, f_2),\\[10pt]
		\partial_t f_2 + v \cdot \nabla_x f_2= \displaystyle \frac{1}{\kappa} Q^{21} (f_2, f_1) +\frac{1}{\eps} Q^{22}(f_2, f_2,
	\end{cases}
\end{equation} 
where $\eps>0$ and $\kappa>0$ may not have the same order of magnitude. Despite the minimum of the entropy (the global equilibrium) is still a pair of global Maxwellian $(\mathcal{M}_1, \mathcal{M}_2)$, the dynamics involves processes at two different time scales and the convergence no longer occurs on the same timescale. Let us assume that the two species are segregated, as in Fig. \ref{Fig_bubbles_interface}. For the sake of illustration, we can think of a setting in which droplets (or bubbles) of light gas are separated by a heavy gas. The separation between the two species leads to a configuration in which inter-species collisions can occur only along the separation layers (\textit{interface}). On the contrary, intra-species collisions can occur inside the volume (in Fig. \ref{Fig_bubbles_interface}, dashed volumes for one species and white for the other species). This corresponds to the regime $\kappa = \mathcal{O}(L/T)$ where $L$ and $T$ are the respective characteristic length and time scales, and $\eps \ll \kappa$, in which intra-species collisions are much more frequent than inter-species ones. In this case, for $\eps \to 0$, each distribution function $f_p$ converges towards the local Maxwellian \eqref{local_Maxwellian} with local average velocity $\bar{v}_p$ and local temperature $T_p$, before reaching a global equilibrium. The process which leads to common average velocities and common temperatures is much slower and we assume it can be disregarded for finite times.

This regime is physically realistic for mixtures of heavy and light gases, usually referred to as $\eps$-mixtures, as described in  \cite{bisi2021hydro, galkin1994}.

\subsection{Macroscopic Quantities}
\label{Section_Macroscopic_Quantities}
Let us define the macroscopic fields for each species as suitable moments of the distribution functions $f_p$. In particular, the number density $n_p$, the mass density $\rho_p$, the average macroscopic velocity $\bar{u}_p$, the partial pressure $P_p$ and the temperature $T_p$ for each species are computed as
\begin{equation}
	n_p=\int_{\R^{d}} f_p(\uv) d\uv, \quad \rho_p \bar{u}_p=\int_{\R^{d}} m_p \uv f_p d\uv ,\quad T_p=\frac{m_p}{d \, n_p} \int_{\R^{d}} |\uv-\bar{u}_p|^2 f_p d\uv,
\end{equation}
with $\rho_p= m_p\, n_p$ and $P_p= n_p \, T_p$. We can also define the total moments for the mixture, given by
\begin{multline}
\label{total_macro_quantities}
	n= \sum_p n_p, \quad \rho= \sum_p \rho_p, \quad \rho\bar{u} = \sum_p \rho\bar{u}_p,\\ P=\sum_p P_p, \quad T= \sum_p \frac{m_p}{d \, n_p} \int_{\R^{d}} |\uv-\bar{u}|^2 f_p d\uv.
\end{multline}
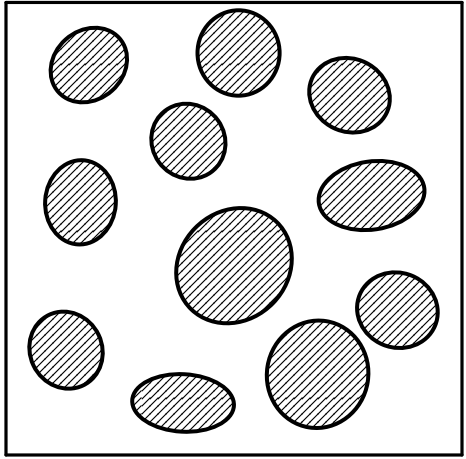
\begin{figure}[ht!]
\begin{center}
\begin{tikzpicture}[line cap=round,line join=round,>=triangle 45,x=1cm,y=1cm]
\fill[line width=1.5pt,color=white] (-3,3) -- (-3,-3) -- (3,-3) -- (3,3) -- cycle;
\draw [line width=1.2pt] (-3,3)-- (-3,-3);
\draw [line width=1.2pt] (-3,-3)-- (3,-3);
\draw [line width=1.2pt] (3,-3)-- (3,3);
\draw [line width=1.2pt] (3,3)-- (-3,3);
\draw [rotate around={42.13759477388822:(-1.91,2.17)},line width=1.5pt, fill=black,pattern=north east lines,pattern color=black] (-1.91,2.17) ellipse (0.5370713995679743cm and 0.45633944409167876cm);
\draw [rotate around={-26.565051177077933:(1.52,1.77)},line width=1.5pt, fill=black,pattern=north east lines,pattern color=black] (1.52,1.77) ellipse (0.54128762003522cm and 0.4821745406005922cm);
\draw [rotate around={-54.462322208025896:(-0.6,1.16)},line width=1.5pt, fill=black,pattern=north east lines,pattern color=black] (-0.6,1.16) ellipse (0.507454728901928cm and 0.47739952019763227cm);
\draw [rotate around={86.30861401354811:(-2.02,0.35)},line width=1.5pt, fill=black,pattern=north east lines,pattern color=black] (-2.02,0.35) ellipse (0.559016994374947cm and 0.46475800154488955cm);
\draw [rotate around={8.583621480113953:(1.81,0.44)},line width=1.5pt, fill=black,pattern=north east lines,pattern color=black] (1.81,0.44) ellipse (0.7008191611701997cm and 0.45149473603055706cm);
\draw [rotate around={46.080924186660795:(0,-0.49)},line width=1.5pt, fill=black,pattern=north east lines,pattern color=black] (0,-0.49) ellipse (0.8068454425853496cm and 0.7144925249579236cm);
\draw [rotate around={-64.65382405805283:(-2.21,-1.61)},line width=1.5pt, fill=black,pattern=north east lines,pattern color=black] (-2.21,-1.61) ellipse (0.5186165188286124cm and 0.4740918619865859cm);
\draw [rotate around={80.90972307917617:(1.1,-1.93)},line width=1.5pt, fill=black,pattern=north east lines,pattern color=black] (1.1,-1.93) ellipse (0.7133424680722854cm and 0.6669013995752738cm);
\draw [rotate around={-22.833654177917516:(2.15,-1.08)},line width=1.5pt, fill=black,pattern=north east lines,pattern color=black] (2.15,-1.08) ellipse (0.5364905321312765cm and 0.49530000107662014cm);
\draw [rotate around={-3.122130462115711:(-0.67,-2.31)},line width=1.5pt, fill=black,pattern=north east lines,pattern color=black] (-0.67,-2.31) ellipse (0.6698345166751025cm and 0.38115387933139233cm);
\draw [rotate around={90:(0.06,2.33)},line width=1.5pt, fill=black,pattern=north east lines,pattern color=black] (0.06,2.33) ellipse (0.5662075462217281cm and 0.5400842391686972cm);
\end{tikzpicture}
\caption{Diffused interface gases in a control volume.}
\label{Fig_bubbles_interface}
\end{center}
\end{figure}

All these quantities are valid for situations in which the two species are not well separated. When we deal with
gases in which species are segregated, intra-species collisions become more frequent and different regimes must be taken into account. For this purpose, we introduce a new quantity $\alpha_p$, describing the volume fraction occupied by the species $p$. Heuristically, let us consider a control volume $V$ and let $\chi_p(x,t)$ be the probability of finding a particle of the species $p$ in $(x,t)$. Then, the volume fraction $\alpha_p(x,t)$ occupied by the phase $p$ is defined as
\begin{equation}
\label{volume_fraction}
    \alpha_p(x,t) = \frac{1}{|V|} \int_V \chi_p(x,t)\, dV.
\end{equation}
In this setting, $\sum_p \alpha_p = 1$ and the situation is depicted in Fig. \ref{Fig_bubbles_interface} for a mixture of two species.
\begin{rem}
The role of the volume fraction for each species is fundamental for models of separated gases \cite{baer1986, toro2010}. Indeed, in the passage to the macroscopic level, the system could be described using volume fractions that smoothly vary from one region to the other in the space scale of the control volume $V$.
\end{rem}

\section{Hydrodynamic Limit for Multi-species Boltzmann Equation}
\subsection{The single-velocity case}
In this section we describe the formal derivation of the Euler equations from the multi-species Boltzmann equation, under the assumption that the Knudsen number is approximately the same for all species. In the limit of small Knudsen numbers, that is $\eps \to 0$, the frequency of collisions increases to infinity and the solution can be completely described by its local hydrodynamic fields. To do so, we recall the multi-species Boltzmann equation rescaled according to the hyperbolic scaling, where all the collision operators scale with the same rate
\begin{equation}
	\label{BEs_hyperbolic_multi}
	\partial_t f_p + \uv \cdot \nabla_x f_p = \frac{1}{\eps} \mathcal{Q}_p(\f) \qquad p=1,\dots,M
\end{equation}
and we apply the Chapman-Enskog expansion as in \cite{bisi2021hydro}. Formally, we can consider an expansion of the distribution functions in term of $\eps$ as
\begin{equation}
	f_p = f_p^0 + \eps f_p^1 + \mathcal{O}(\eps^2),
\end{equation}
obtaining expansions also for the other quantities of interest. In particular the number densities $n_p$, mass velocities $u_p$ and temperatures $T_p$ can be expanded as
\begin{equation}
	n_p=n_p^0+\eps n_p^1 + \mathcal{O}(\eps^2), \quad \bar{u}_p=\bar{u}_p^0 + \eps \bar{u}_p^1 + \mathcal{O}(\eps^2), \quad T_p=T_p^0 + \eps T_p^1 + \mathcal{O}(\eps^2).
\end{equation}
Let us consider the weak form of the Boltzmann equation truncated at the zeroth order of expansion
\begin{equation}
	\label{weak_hydro}
	\begin{aligned}
		\frac{\partial}{\partial t} \int_{\R^d} \varphi(\uv) (f_p^0 + \eps f_p^1) d\uv \,+\, &\nabla_x \cdot \int_{\R^d} \varphi(\uv) \uv (f_p^0 + \eps f_p^1) d\uv \\ &= \frac{1}{\eps} \sum_{q=1}^N \int_{\R^d} \varphi(\uv) Q^{pq} (f_p^0 +\eps f_p^1,f_q^0 +\eps f_q^1) d\uv,
	\end{aligned}
\end{equation}
where $\varphi(\uv)$ is a $\mathcal{C}^\infty$ test function.\\

\begin{rem}
In the Chapman-Enskog expansion, $f_p^1$ designates a quantity that depends smoothly on the macroscopic quantities and any finite number of its derivatives with respect to the $(x,v)$-variables (but not on time). In particular, $f_p^1$ satisfies
\begin{equation}
\label{properties_conservation_f1p}
    \int_{\R^d} f_p^1 \begin{pmatrix}
    1\\
    v\\
    \frac{1}{2}|v|^2
    \end{pmatrix} \, dv = 0.
\end{equation}
\end{rem}
The limit for $\eps \to 0$ in \eqref{BEs_hyperbolic_multi} implies naturally a vanishing leading order collision operator, thus $f_i^0$ are the Boltzmann equilibria \eqref{global_maxwellian}.  The $0$-th order term of equation \eqref{weak_hydro} provides the multi-species (or multi-component) Euler equations for macroscopic quantities and the first order correction corresponds to the multi-species (or multi-component) Navier-Stokes system \cite{arnault2022chapman}. In other words, we can claim that formally the solution to the multi-species Euler equations and the solution to the multi-species Navier-Stokes equations are close to the moments of the solution $f$ to the multi-species Boltzmann equation respectively at order $\mathcal{O}(\eps)$ and $\mathcal{O}(\eps^2)$.

The multi-species Euler system can be directly recovered with a different approach, considering the collision invariants as test functions in the weak formulation
\begin{equation*}
	\frac{\partial}{\partial t} \int_{\R^d} \varphi(\uv) f_p d\uv + \nabla_x \cdot \int_{\R^{d}} \varphi(\uv) \uv f_p \,d\uv = \sum_{q=1}^N \int_{\R^d} \varphi(\uv) Q^{pq} (f_p,f_q) d\uv.
\end{equation*}
Let us consider $\varphi(\uv) \equiv 1$. Since it corresponds to a collision invariant of the Boltzmann operator, the right-hand side vanishes and it leads to the equation
\begin{equation}
	\frac{\partial n_p}{\partial t} + \nabla_x \cdot (n_p \bar{u}) = 0, \quad p=1,\dots, M,
\end{equation}
where we have used $\bar{u}_p = \bar{u}$ for all $p$.
Multiplying each equation for the respective mass $m_p$ we obtain
\begin{equation}
	\label{hydro_partial_density}
	\frac{\partial \rho_p}{\partial t} + \nabla_x \cdot (\rho_p \bar{u}) = 0, \quad p=1,\dots, M.
\end{equation}
Finally, summing over the index $p$, the continuity equation for the total density holds
\begin{equation}
	\label{hydro_total_density} 
	\frac{\partial \rho}{\partial t} + \nabla_x \cdot (\rho \bar{u}) = 0.
\end{equation}
Now consider $\varphi(\uv)=m_p \uv$ and sum over the index $p$, in order to obtain
\begin{equation}
	\label{hydro_total_momentum}
	\frac{\partial}{\partial t} (\rho \bar {u}) + \nabla_x \cdot \Big[\displaystyle \sum_{p=0}^P \Big( \rho_p \bar{u} \otimes \bar{u} + n_p \, T_{eq}\Big)\Big] = 0,
\end{equation}
where the single-species temperature $T_p = T_{eq}$. 
Let us consider $\varphi(\uv)=m_p |\uv|^2$ and sum again over the index $p$. Thus, we obtain
\begin{equation}
	\frac{\partial}{\partial t} \left(\frac{1}{2}\rho |\bar {u}|^2 + \frac{d}{2} n \, T_{eq} \right) + \nabla_x \cdot \left[\displaystyle \left( \frac{1}{2}\rho |\bar{u}|^2 + \frac{d+2}{2} n \, T_{eq} \right)\bar{u}\right] = 0.
\end{equation}
Thus, in the limit $\eps \to 0$, we have formally obtained the so-called multi-species Euler limit:
\begin{equation}
	\label{multispecies_euler}
	\begin{cases}
		\displaystyle \frac{\partial \rho_p}{\partial t} + \nabla_x \cdot (\rho_p \bar{u}) = 0, \quad p=1,\dots, M,\\[10pt]
		\displaystyle \frac{\partial}{\partial t} (\rho \bar {u}) + \nabla_x \cdot \left[ \sum_{p=1}^N \rho_p \left( \bar{u} \otimes \bar{u} + n_p \, T_{eq} \right) \right] = 0,\\[18pt]
		\displaystyle \frac{\partial}{\partial t} \left(\frac{1}{2}\rho |\bar {u}|^2 + \frac{d}{2} n\,T_{eq} \right) + \nabla_x \cdot \left[\displaystyle \left( \frac{1}{2}\rho |\bar{u}|^2 + \frac{d+2}{2}n\,T_{eq} \right)\bar{u}\right] = 0,
	\end{cases}
\end{equation}
with a single velocity and a single temperature for the whole mixture.

\subsection{The multi-velocity case}
Now, we are interested in describing the formal hydrodynamic limit in a different regime, in order to recover the isentropic two-phase flow model proposed in \cite{toro2010}, which in the one-dimensional case reads as
\reqnomode
\begin{subnumcases}
	\displaystyle \frac{\partial \rho}{\partial t} + \frac{\partial}{\partial x} (\rho \bar{u}) = 0, \label{RDT_1}\\[10pt]
	\displaystyle \frac{\partial}{\partial t}(\alpha_1 \rho) + \frac{\partial}{\partial x} (\alpha_1 \rho \bar{u}) = -\frac{1}{\tau}(P_2-P_1), \label{RDT_2}\\[10pt]
	\displaystyle \frac{\partial}{\partial t} (\alpha_1 \rho_1) + \frac{\partial}{\partial x} (\alpha_1 \rho_1 \bar{u}_1) = 0, \label{RDT_3}\\[10pt]		
	\displaystyle \frac{\partial}{\partial t} (\rho \bar{u} ) + \frac{\partial}{\partial x} \Big(\alpha_1 \rho_1 \bar{u}_1^2 + \alpha_1 P_1 + \alpha_2 \rho_2 \bar{u}_2^2 + \alpha_2 P_2 \Big) = 0, \label{RDT_4}\\[10pt]
	\displaystyle \frac{\partial}{\partial t} (\bar{u}_1-\bar{u}_2) + \frac{\partial}{\partial x} 
	\left(\frac{1}{2}\bar{u}_1^2 + h_1 -\frac{1}{2}\bar{u}_2^2 -h_2 \right) = -\zeta (\bar{u}_1 - \bar{u}_2), \label{RDT_5}
\end{subnumcases}
where $\alpha_p \in (0,1)$ is the volume fraction defined in \eqref{volume_fraction}, $\tau$ and $\zeta$ are opportune constants, and the partial pressures $P_p$ will be described later. In this framework, the total density $\rho$ is defined as 
\begin{equation}
    \rho = \alpha_1\,\rho_1 + \alpha_2\,\rho_2,
\end{equation}
and the mixture velocity $\bar{u}$ as
\begin{equation}
\label{mixture_velocity}
    \bar{u} = \kappa_1\,\bar{u}_1 + \kappa_2\,\bar{u}_2, \qquad \kappa_p = \frac{\alpha_p\,\rho_p}{\rho}.
\end{equation}
Compared to the model of two-phase mixtures proposed by Baer and Nunziato in \cite{baer1986}, the main advantage of this alternative model is the possibility to write the equations in a conservative form.

Equations \eqref{RDT_1} and \eqref{RDT_3} describe the conservation of the total mass and the partial mass, respectively. Equation \eqref{RDT_2} consists of the balance of the evolution of the volume fraction with respect to a pressure relaxation term. Further, Equation \eqref{RDT_4} gives the conservation of the total momentum $\rho\bar{u}$ and Equation \eqref{RDT_5} the balance of the relative velocity $\bar{u}_1-\bar{u}_2$ with respect to a friction term.\\
Since the system \eqref{RDT_1}-\eqref{RDT_5} is isentropic, we set $P_p = P(\rho_p) = c \, \rho_p^\gamma$. This closure law for the pressure terms implies that an equation for the second moment is not required, differently from \eqref{multispecies_euler}.

To obtain the hyperbolic two-phase flow model as a limit system from a kinetic model, we take into account the case of a gas mixture as in Fig. \ref{Fig_bubbles_interface}. We recall that in this setting, self-collisions are much more frequent than collisions among particles of different species, namely $\kappa=1$ and $\eps \ll 1$ in \eqref{system_2knudsen}, as described in Section \ref{Section_Macroscopic_Quantities}. For this purpose, we then consider the following rescaling of the multi-species Boltzmann equation
\begin{equation}
	\label{BEs_multi_regime}
	\begin{cases}
		\partial_t f_p + \uv \cdot \nabla_\x f_p = \displaystyle \frac{1}{\eps} Q^{pp}(f_p,f_p) + \sum_{\substack{q=1 \\ q \neq p} }^M Q^{pq} (f_p,f_q),\\ 
		f_p(0,\x,\uv)= f_{p,0} (\x,\uv).
	\end{cases}
\end{equation}
The weak form of the Boltzmann equation in this scaling becomes:
\makeatletter
\tagsleft@true
\makeatother
\begin{multline}
\label{weak_noaverage}
	\frac{\partial}{\partial t} \int_{\R^d} \varphi(v) f_p d\uv + \nabla_x \cdot \int_{\R^{d}} \varphi(\uv) \uv f_p d\uv \\ = \frac{1}{\eps} \int_{\R^d} \varphi(\uv) Q^{pp}(f_p,f_p) d\uv + \sum_{\substack{q=1 \\ q \neq p}}^M \int_{\R^d} \varphi(\uv) Q^{pq} (f_p,f_q) d\uv.
\end{multline}
Next, we average the weak form over the control volume $V$:
\begin{multline*}
	\int_V \chi_p \, \frac{\partial}{\partial t} \int_{\R^d} \varphi(v) f_p d\uv\, dV + \int_V \, \chi_p \nabla_x \cdot \int_{\R^{d}} \varphi(\uv) \uv f_p d\uv\, dV \\ = \frac{1}{\eps} \int_V \chi_p \int_{\R^d} \varphi(\uv) Q^{pp}(f_p,f_p) d\uv\, dV + \sum_{\substack{q=1 \\ q \neq p}}^M \int_V \chi_p \int_{\R^d} \varphi(\uv) Q^{pq} (f_p,f_q) d\uv\, dV,
\end{multline*}
where $\chi_p$ is the probability to find a particle of species $p$ in $(x,t)$. Each $\chi_p$ is associated to the corresponding volume fraction $\alpha_p$, as introduced in \eqref{volume_fraction}.
 Averaging procedure are widely used in the derivation of multi-phase models from the corresponding single-phase formulations (e.g. from multi-component Euler systems), see for instance \cite{drew1998, abgrall2003}. In this work, the spatial averaging procedure is instead directly applied to the kinetic equation, following the same rules presented in \cite{drew1998}. Indeed, the averaging procedure over a control volume $V$ commutes with time and space derivatives, and the following calculation rules hold for any function $g$:
\setlist[itemize,1]{label={\small \textbullet}}
\begin{itemize}
    \item Gauss rule
        \begin{equation*}
            \int_V \chi_p \nabla_x g \, dV =  \int_V \nabla_x (\chi_p\,g) \, dV - \int_{\partial V} g_{\text{interface}}\, \nabla_x \chi_p \, dV,
        \end{equation*}
    \item Leibniz rule
        \begin{equation*}
            \int_V \chi_p \partial_t g \, dV =  \int_V \partial_t (\chi_p\,g) \, dV - \int_{\partial V} g_{\text{interface}}\, \partial_t \chi_p \, dV,
        \end{equation*}
\end{itemize}
where $g_{\text{interface}}$ is the value of $g$ at $\partial V$. The boundary $\partial V$ of the domain includes both the external physical boundary and all interfaces separating the phases, where state variables may exhibit discontinuities. Let us then define macroscopic fields as moments of the distribution functions averaged over a unit control volume $V$ (avoiding the scaling factor $|V|^{-1}$ for the sake of simplicity) according to the quantity $\chi_p$, namely
	\begin{equation}
		\int_V \int_{\R^d} \chi_p m_p \, f_p(\uv) d\uv\,dV=\alpha_p \rho_p, \qquad \int_V \int_{\R^d} \chi_p m_p \, \uv \, f_p(\uv) d\uv\,dV=\alpha_p \rho_p \bar{\uv}_p,
	\end{equation}
where we recall that $\alpha_p$ stands for the volume fraction. This choice allows us to give a kinetic interpretation of the macroscopic quantities involved in \eqref{RDT_1}-\eqref{RDT_5}.
The probability $\chi_p$ evolves according to the equation 
\begin{equation}
\label{evolution_chi}
    \partial_t \chi_p + \bar{u}_{\text{interface}} \cdot \nabla_x \chi_p =0,
\end{equation}
where $\bar{u}_{\text{interface}}$ is the interface velocity between the two phases.
\begin{rem}
The spatial average procedure is fundamental to obtain equations for different phases, describing ``non-local'' processes which happen at the scale of the control volume.
\end{rem}
In order to derive the equation for the fraction $\alpha_1$, we have to determine an expression for the transfer of volume in the dynamics. We will show that the volume transfer is caused by the difference in pressures between the two gases, which in the literature is usually assumed as a phenomenological hypothesis \cite{baer1986, saurel1999, flatten2011}.

\begin{center}
	\begin{figure}
		\begin{tikzpicture}[line cap=round,line join=round,>=triangle 45,x=1cm,y=1cm, scale=1.3]
		\begin{scope}[opacity=0.3]
            \fill[pattern=north east lines, pattern color=black] (2.5,0) rectangle (5,5);
        \end{scope}
			\draw[line width=0.8pt] (0,0) -- (0,5) -- (5,5) -- (5,0) -- cycle; 
			\draw [line width=0.8pt] (0,0)-- (0,5);
			\draw [line width=0.8pt] (0,5)-- (5,5);
			\draw [line width=0.8pt] (5,5)-- (5,0);
			\draw [line width=0.8pt] (5,0)-- (0,0);
			\draw [line width=0.8pt] (2.5,0)-- (2.5,5);
			\draw [line width=0.8pt,dash pattern=on 1pt off 1pt] (1.54634,2.62442)-- (3.45366,2.62442);
			\draw (0.14,5.02) node[anchor=north west] {\large $\mathbf{1}$};
			\draw (4.56,5.02) node[anchor=north west] {\large $\mathbf{2}$};
			\draw [ultra thick, latex'-latex', line width=0.6pt] (1.54634,2.12442)-- (2.46,2.12442);
			\draw [ultra thick, latex'-latex', line width=0.6pt] (3.44634,2.12442)-- (2.54,2.12442);
			\begin{small}
				\draw [fill=black] (1.54634,2.62442) circle (1pt);
				\draw[color=black] (1.25,2.88) node { $x_1(t)$};
				\draw [fill=black] (3.45366,2.62442) circle (1pt);
				\draw[color=black] (3.81,2.88) node { $x_2(t)$};
				\draw [fill=black] (2.5,2.62442) circle (1pt);
				\draw[color=black] (2.82,2.88) node {$x(t)$};
				\draw[color=black] (2.09,1.88) node { $\eta_1(t)$};
				\draw[color=black] (3.09,1.88) node { $\eta_2(t)$};
			\end{small}
		\end{tikzpicture}
		\caption{Interface between two different species.}
		\label{figure_interface}
	\end{figure}
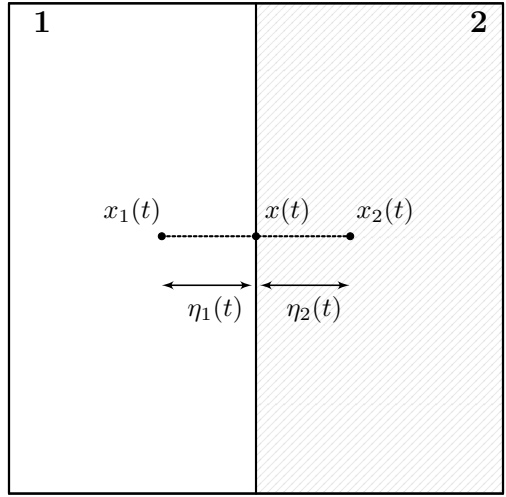
\end{center}

Let us consider a zoom of the setting of Fig. \ref{Fig_bubbles_interface} on the thin layer which separates the volume occupied by one species from the volume occupied by the other, as shown in Fig. \ref{figure_interface}.We now focus on the displacement of the interface separating the two species. Following the idea of Bresch, Burtea and Lagoutière in \cite{bresch2023}, at this scale the volume fraction variation can be reinterpreted as a one-dimensional shift orthogonal to the separating interface. Denoting with $x(t)$ the average distance of particles from the interface of separation between the two species, let us consider $\eta_1(t):=x_1(t)-x(t)$ and $\eta_2(t):=x(t)-x_2(t)$ where $x_1$ and $x_2$ are the local centers of mass of the species $1$ and $2$ respectively. We indeed perform a spatial average over a thin layer which is much smaller than $V$, so we refer to the ``centers of mass'' of the molecules within this tiny volume.

Assuming that $\eta_p \ll \text{diameter}(V)$, let us define
\begin{equation}
	\alpha_1(t)=\frac{\eta_1(t)}{\eta_1(t)+\eta_2(t)}.
\end{equation}
Let us introduce the material derivative  as 
\begin{equation}
	\frac{D}{Dt} := \frac{\partial}{\partial t} + \bar{u} \cdot \nabla_x,
\end{equation}
where we recall that $\bar{u}$ is the mixture velocity. Since we are considering segregated gases, the mixture velocity satisfies $\bar{u} = \bar{u}_p$ in the region occupied only by species $p$.\\
In the reference frame of the interface, the material derivative of the volume fraction $\alpha_1$ is given by
\begin{equation}
	\frac{D}{Dt} \alpha_1(t) = \frac{D_t \eta_1 (\eta_1+\eta_2) - \eta_1 (D_t \eta_1 + D_t \eta_2)}{(\eta_1+\eta_2)^2},
\end{equation} 
where we have omitted the explicit dependence of $\eta_p$ on time. As we considered a small shift with respect to the interface, we have
\begin{equation}
	D_t \eta_p(t) = \bar{u}_p(t,x_p(t)) - \bar{u}(t,x(t)).
\end{equation}
Therefore, the expression of the derivative of the volume fraction becomes
\begin{equation}
	\label{volume_fraction_derivative}
	\begin{aligned}
		\frac{D}{Dt} \alpha_1(t) &= \frac{(\bar{u}_1(t,x_1(t))-\bar{u}(t,x(t))) (\eta_1+\eta_2) - \eta_1 (\bar{u}_1(t,x_1(t))-\bar{u}_2(t,x_2(t))) }{(\eta_1+\eta_2)^2}\\
		&=\frac{\eta_1 (\bar{u}_2(t,x_2(t))-\bar{u}(t,x(t)))+ \eta_2 (\bar{u}_1(t,x_1(t))-\bar{u}(t,x(t)))}{(\eta_1+\eta_2)^2}
.	\end{aligned}
\end{equation}
Let us assume that across the diffuse interface there is a ``boundary layer'' really separating the two gases, which provides the existence of a continuous stress tensor at the interface: namely given a coefficient $\lambda$ representing the kinematic viscosity, the quantity
\begin{equation}
\label{continuity_stress_tensor_P}
	P(t,x(t))-\sqrt{\eps}\,\lambda \partial_x \bar{u}(t,x(t))
\end{equation}
is continuous at the interface. 
The derivation of \eqref{continuity_stress_tensor_P} at the boundary will be detailed in the separate Section \ref{Derivation_Interface} for the reader's convenience.

Approximating the derivative at first order with respect to $\eta_1$ and $\eta_2$ respectively, we have
\begin{multline}
	\displaystyle P(t,x_1(t)) - \sqrt{\eps} \, \lambda \frac{\bar{u}_1(t,x_1(t))-\bar{u}(t,x(t))}{\eta_1(t)}  \\ = P(t,x_2(t)) - \sqrt{\eps} \, \lambda \frac{\bar{u}(t,x(t))-\bar{u}_2(t,x_2(t))}{\eta_2(t)} + \mathcal{O}(\eta_1+\eta_2).
\end{multline}
Finally, substituting this expression in \eqref{volume_fraction_derivative} and considering the limit for $\eta_1+\eta_2 \to 0$, we obtain
\begin{equation}
\label{eq_diff_pressure}
	P_1-P_2 = \lambda \sqrt{\eps} \frac{(\eta_1+\eta_2)^2}{\eta_1\eta_2} D_t \alpha_1(t).
\end{equation}
Now, let us compute
\begin{equation}
	\frac{\partial}{\partial t} (\alpha_1 \rho) + \nabla_x (\alpha_1 \rho \bar{u}) = \alpha_1 \frac{\partial \rho}{\partial t} + \rho \frac{\partial \alpha_1}{\partial t} + \alpha_1 \nabla_x (\rho \bar{u}) + (\rho \bar{u}) \cdot \nabla_x (\alpha_1).
\end{equation}
Recalling the definition of the material derivative and using Equations \eqref{hydro_total_density}-\eqref{eq_diff_pressure}, we obtain
\begin{equation}
\label{relation_pressure_volume}
	\frac{\partial}{\partial t} (\alpha_1 \rho) + \nabla_x \cdot (\alpha_1 \rho \bar{u}) = \rho \frac{D}{Dt} (\alpha_1) = - \frac{1}{\tau} ( P_2-P_1),
\end{equation}
where $\tau := \displaystyle \frac{\sqrt{\eps} \lambda}{\rho} \frac{(\eta_1+\eta_2)^2}{\eta_1\,\eta_2}$ is the relaxation rate.\\
Thus, the pressure relaxation process drives a volume variation which evolves according to \eqref{relation_pressure_volume}.
In particular, we observe that the pressure relaxation implies a volume transfer towards the gas with highest pressure, according to the relaxation parameter $\tau >0$.
\begin{rem}
The macroscopic velocity $\bar{u}$ for each point $(x,t)$ is given by \eqref{mixture_velocity}, obtained from the kinetic equation as \eqref{total_macro_quantities}. Inside a region occupied only by one species $p$, the macroscopic velocity automatically reduces to the single-species macroscopic velocity $\bar{u}_p$.
\end{rem}

Then, since we have conservation of the single-species mass and conservation of total momentum at the kinetic level, the derivation of \eqref{RDT_3}-\eqref{RDT_4} is obtained taking respectively $\varphi(\uv)=1, \, m_p \uv$, averaging in space and summing over the index $p$, using Gauss and Leibniz rules and \eqref{relation_pressure_volume}. We observe that, despite setting the number of species $M$ to $2$, the validity of these equations remains independent from the number of species. The equation for total mass conservation \eqref{RDT_1} is a straightforward consequence of the conservation of single-species mass.\\
Finally, we would like to recover the equation for the relative velocity. Considering the test function $\varphi(\uv)=m_p \uv$ in the weak formulation \eqref{weak_noaverage}, we obtain 
\begin{equation}
	\label{relative_velocity_eq}
	\frac{\partial}{\partial t} (\rho_p \bar {u}_p) + \nabla_x \cdot \Big( \rho_p \bar{u}_p \otimes \bar{u}_p + P_p\Big) = \sum_{q \neq p} \mathbf{R}_{p,q},
\end{equation}
where $\mathbf{R}_{p,q}$ is the momentum exchange rate. Following the idea presented in \cite{bisi2012multi}, we shall compute explicitly $\mathbf{R}_{p,q}$ in the case of Maxwellian molecules and hard spheres molecules. To this aim, in Section \ref{Subsection_Computation_Exchange} we compute the product of two local Maxwellians 
$\mathcal{M}_p(v) \, \mathcal{M}_q(w)$, whose expression depends on the center of mass $\mathcal{V}$ and relative velocities $|v-w|$. Then, assuming pseudo-Maxwellian molecules (constant kernel $\mathbf{B}$ independent from the species) or hard-spheres molecules (kernel of the form $\mathbf{B}=|u|$ independent from the species), the explicit computation described in \ref{Subsection_Computation_Exchange} leads to the following expression for the momentum exchange rate:
\begin{equation}
	\label{bisi_exchange}
	\mathbf{R}_{p,q}=-\xi_{p,q} \rho_p \rho_q (\bar{u}_p-\bar{u}_q), \quad \xi_{p,q} \in \R,
\end{equation}
where the coefficients $\xi_{p,q}$ satisfy the symmetric property $\xi_{p,q}=\xi_{q,p}$. We first notice that the computation can be done for general distribution function $f_p$ and $f_q$, if we are considering pseudo-Maxwellian molecules. Indeed, for this particular choice, the interaction kernel does not depend on the deviation angle nor the relative velocity and the exchange rate can be written using \eqref{weak_bispecies} and the expression for the moments as
	\begin{equation}
		\mathbf{R}_{p,q} = \int_{\R^d \times \R^d \times \mathbb{S}^{d-1}} m_p m_q f_p(v) f_q(v_*) (v-v_*) d\sigma dv\,dv_* =   \rho_p \rho_q (\bar{u}_p-\bar{u}_q).
	\end{equation} 
	For general interaction particle, the presence of the interaction kernel make the computation really insidious and their explicit expressions are given in the Appendix \ref{Subsection_Computation_Exchange}.\\
Now, we can rewrite explicitly the equation \eqref{relative_velocity_eq} and subtract the equation for the other species, obtaining
\begin{multline}
	\frac{1}{\rho_1} \frac{\partial}{\partial t} (\rho_1 \bar{u}_1) - \frac{1}{\rho_2} \frac{\partial}{\partial t} (\rho_2 \bar{u}_2) \\ + \frac{1}{\rho_1} \nabla_x \cdot \left(\rho_1 \, \bar{u}_1 \otimes \bar{u}_1 + P_1\right) - \frac{1}{\rho_2} \nabla_x \cdot  \left(\rho_2 \bar{u}_2 \otimes \bar{u}_2 + P_2\right) \\ = \frac{\mathbf{R}_{1,2}}{\rho_1}-\frac{\mathbf{R}_{2,1}}{\rho_2} = -\xi \left(\rho_2+\rho_1 \right) \,(\bar{u}_1-\bar{u}_2).
\end{multline}
We can rewrite the pressure term in terms of the internal energy and the phase enthalpy as
\begin{equation}
	\label{pressure_hydro}
	P_p = \rho_p ( h_p -e_p),
\end{equation}
where $h_p$ is the phase enthalpy and $e_p$ is the internal energy. Assuming a polytropic gas for which $P(\rho)=\rho^\gamma$, the following relation holds
\begin{equation}
\label{pressure_enthalpy_derivative}
    \partial_x P_p = \rho_p \partial_x h_p.
\end{equation}
Indeed, for an ideal gas it holds $P = \rho\,e (\gamma-1)$
\begin{equation}
    \frac{1}{\rho} \nabla_x P = \frac{\gamma}{\rho}\, \rho^{\gamma-1} \nabla_x \rho, \qquad \partial_x h = \nabla_x \left(\frac{P}{\rho} + e \right) = \gamma \rho^{\gamma-2} \nabla_x \rho,
\end{equation}
which leads to \eqref{pressure_enthalpy_derivative}.\\
Therefore, using \eqref{bisi_exchange} and \eqref{pressure_hydro} and defining the quantity
\begin{equation}
	\zeta := \xi \, (\rho_1+\rho_2)
\end{equation} 
we recover the following equation in the one-dimensional case
\begin{equation}
		\frac{\partial}{\partial t} ( \bar{u}_1 - \bar{u}_2) + \frac{\partial}{\partial x} \left(\frac{1}{2}\bar{u}_1^2 - \frac{1}{2}\bar{u}_2^2 + h_1 - h_2 \right)= - \zeta (\bar{u}_1-\bar{u}_2).
\end{equation}
Therefore, through a formal hydrodynamic limit of the multi-species Boltzmann equation when $\eps \to 0$ we have obtained a model similar to the two-phase flow model introduced in \cite{toro2010}.

\subsubsection{Derivation of the continuous stress tensor at the interface}
\label{Derivation_Interface}
The idea is to obtain an expression for the stress tensor at the boundary separating two phases, inspired by the the compressible boundary layer equations derived from the Boltzmann equation in \cite{castulik2002} (see also \cite{golse2011} for the incompressible case).\\
Let us focus on the thin boundary layer across the interface of Fig. \ref{figure_interface} and let us introduce a Prandtl scaling. From Prandtl theory, we know that if the Knudsen number is $\mathcal{O}(\eps)$ and the fluid is compressible, then the thickness of such a region can be assumed $\mathcal{O}(\sqrt{\eps})$. For the sake of simplicity, we can assume that the interface is located at $x_D=0$ and the species $1$ moves in the region described by the half space $x'=(x_1,\dots, x_{D-1})$ with $x_D<0$. Let us rescale the independent variables by obtaining:
\begin{equation}
    \x'=\x, \qquad x'_D=\frac{x_D}{\sqrt{\eps}}, \qquad \text{ and } \qquad t'=t.
\end{equation}
The rescaled Boltzmann equation reads as
\begin{equation}
\label{Rescaled_BEs}
    \partial_t f_p + \uv' \cdot \nabla_{x'} f_p + \frac{1}{\sqrt{\eps}} v_D \partial_{x_D} f_p = \frac{1}{\eps} \mathcal{Q}_{pp}(f_p,f_p) + \mathcal{Q}_{pq}(f_p,f_q),
\end{equation}
where the advection term in the $D$-direction describes the effect of the boundary layer at the interface. Since the process is ``local'' across the interface, no spatial averaging procedure is performed.\\
By using the Chapman-Enskog expansion introduced before in the weak formulation of \eqref{Rescaled_BEs} and matching different powers of $\eps$, the truncation at the $0$-th order level needs $\mathcal{Q}_{pp}=0$ (which implies $f_p$ must be a local Maxwellian) and
\begin{equation}
\label{eq:sqrteps:chapman}
    \frac{1}{\sqrt{\eps}} \int_{\R^d} \varphi(v)\, v_D \, \partial_{x_D}\,\left(f_p^0 + \eps f_p^1\right)\, dv = 0.
\end{equation}
Since \eqref{eq:sqrteps:chapman} is the only $1/\sqrt{\eps}$ term, this implies
\begin{equation}
    \int_{\R^d} \varphi(v)\, v_D \, \partial_{x_D}\,\left(f_p^0 + \eps f_p^1\right)\, dv = 0.
\end{equation}
Let us first take $\varphi(v) = m_p$ in \eqref{eq:sqrteps:chapman} and sum over the species $p$, from which using \eqref{properties_conservation_f1p} we have
\begin{equation}
\label{cons_mass_prandtl}
    \partial_{x_D} \left(\displaystyle \sum_p\rho_p\,\bar{u}_p^D \right) = 0,
\end{equation}
where $u^D$ represents the component of $u$ in the $D$-direction.\\
Then, taking $\varphi(v) = m_p\,v_D$ in \eqref{eq:sqrteps:chapman} and summing over the species $p$, we obtain
\begin{equation}
\label{div_stress}
    \partial_{x_D} \left(\displaystyle \sum_p\rho_p\, (u_p^D)^2 + \displaystyle \sum_p P_p - \displaystyle \sum_p \sqrt{\eps} \lambda \, \partial_{x_D} \bar{u}_p^D - \displaystyle \sum_p \eps \lambda \, \mathcal{D}_{x'}(\bar{u}_p) \right) = 0,
\end{equation}
where $\lambda$ is the proper viscosity coefficient. The components depending on $\sqrt{\eps}$ and $\eps$ are obtained as corrections in $\eps$ in the Chapman-Enskog expansion. Indeed, as shown in \cite{golse2005, saintraymond2009}, the conservation law at order $0$ and the Hilbert's Lemma leads to a first-order flux correction of the form $\eps\,\text{div}_x (\lambda \mathcal{D}(\bar{u}_p))$, where 
\begin{equation}
    \mathcal{D}(u) = \frac{1}{2} \left ( \nabla_x u + (\nabla_x u)^T \right) - \frac{1}{3} \left( \nabla_x \cdot u \right) \text{Id},
\end{equation}
namely $\mathcal{D}(\bar{u}_p)$ is the traceless part of the deformation tensor of $\bar{u}_p$. In the one dimensional case it reduces to $D(u) = \partial_x u$.\\
Let us focus on the expression \eqref{div_stress} at the interface $x_D=0$. We recall that the two species are separated, therefore we can assume without loss of generality that the half-plane $x_D<0$ is occupied by the species $1$ and the half-plane $x_D>0$ is occupied by the species $2$. This implies that the partial momentum $\rho_2\,u_2$ and the partial pressure $P_2$, related to the species $2$, are null in the half-plane $x_D<0$ and viceversa. The first equation \eqref{cons_mass_prandtl} implies
\begin{equation}
    \rho_1(t,x(t))\,\bar{u}_1(t,x(t)) = k =  \rho_2(t,\tilde{x}(t))\,\bar{u}_2(t,\tilde{x}(t)),
\end{equation}
where  $x_D<0$ and $\tilde{x}_D>0$. In other words the interface is a contact discontinuity separating the two gases.\\
By using this relation in the second equation \eqref{div_stress} and truncating at first order in $\eps$, we obtain
\begin{multline}
   k \,(\bar{u}_1(t,x(t)) +  P_1(t,x(t)) - \sqrt{\eps}\,\lambda \frac{\partial \bar{u}_1(t,x(t))}{\partial x_D} \\ = k \,\bar{u}_2(t,\tilde{x}(t)) + P_2(t,\tilde{x}(t)) - \sqrt{\eps}\,\lambda \frac{\partial \bar{u}_2(t,\tilde{x}(t))}{\partial x_D},
\end{multline}
where $x_D<0$ and $\tilde{x}_D>0$. By imposing continuity conditions of the normal velocity at the interface, we finally obtain the continuity of the stress tensor at the boundary at first order in $\eps$, namely
\begin{equation}
   P_1-P_2 = \sqrt{\eps}\, \lambda \partial_{x_D} \bar{u}_1(t,x(t)) - \sqrt{\eps}\, \lambda \partial_{x_D} \bar{u}_2(t,x(t)).
\end{equation}
	
\subsubsection{Comparison with the Baer-Nunziato model}
	Baer-Nunziato-type models \cite{baer1986, abgrall2003}, differently to the model of two-phase compressible flow with two velocities and two pressures proposed in \cite{romenski2004compressible}, are not in a conservative form. From a modeling point of view, the difference between these two approaches relies in the physical description of the phenomena: the conservative model has been proposed to deal with processes in which the thermal state of phases are almost in equilibrium state, while Baer-Nunziato-type models can take into account thermal non-equilibrium processes. To this aim, in this subsection we propose a comparison with the isentropic model derived above, by considering Baer-Nunziato models in which thermal processes are ignored and the phases behavior is isentropic. As observed in \cite{romenski2004compressible, theindumbser2022}, the one-dimensional versions of the models are very similar, if properly rewritten. The main difference is the definition of an interfacial pressure, which provides different results in the modelling of the phenomena. The isentropic version of a simplified Baer-Nunziato model reads as follows
	\reqnomode
    \begin{subnumcases}
    	\displaystyle \frac{\partial}{\partial t}(\alpha_1) + u_{\text{interface}} \partial_x (\alpha_1) = -\frac{1}{\tau}(P_2-P_1), \label{BN_1}\\[10pt]
    	\displaystyle \frac{\partial}{\partial t} (\alpha_1 \rho_1) + \partial_x (\alpha_1 \rho_1 \bar{u}_1) = 0, \label{BN_2}\\[10pt]	
    	\displaystyle \frac{\partial}{\partial t} (\alpha_1 \rho_1 \bar {u}_1) + \partial_x \Big(\alpha_1 \rho_1 \bar{u}_1^2 + \alpha_1 P_1 \Big) = P_{\text{interface}} \partial_x \alpha_1 + \tilde{\zeta} (\bar{u}_2 -\bar{u}_1), \label{BN_3}\\[10pt]
    	\displaystyle \frac{\partial}{\partial t} (\alpha_2 \rho_2) + \partial_x (\alpha_2 \rho_2 \bar{u}_2) = 0, \label{BN_4}\\[10pt]
    	\displaystyle \frac{\partial}{\partial t} (\alpha_2 \rho_2 \bar {u}_2) + \partial_x \Big(\alpha_2 \rho_2 \bar{u}_2^2 + \alpha_2 P_2 \Big) = P_{\text{interface}} \partial_x \alpha_2 - \tilde{\zeta} (\bar{u}_2 -\bar{u}_1),\label{BN_5}
    \end{subnumcases}
    where $u_{\text{interface}}$ and $P_{\text{interface}}$ are  the interfacial velocity and the interfacial pressure respectively. The definition of interfacial velocity is the same for both models 
    \begin{equation}
        u_{\text{interface}} = \bar{u} = \frac{\alpha_1 \rho_1 \bar{u}_1 + \alpha_2 \rho_2 \bar{u}_2}{\alpha_1\rho_1 + \alpha_2\rho_2},
    \end{equation}
    since the first equation \eqref{BN_1} could be easily derived from \eqref{RDT_2}-\eqref{RDT_3}. Equations \eqref{BN_2}-\eqref{BN_4} are derived from the multi-species Boltzmann equation \eqref{BEs_multi_regime}, as already shown for the conservative model.\\
    The main difference lies in the equation for the balance of the single-species momentum. Indeed, by taking $\varphi(v) = m_p\, v$ in \eqref{weak_noaverage}, followed by a spatial averaging procedure, we recover Equations \eqref{BN_3}-\eqref{BN_5}, where $\tilde{\zeta}$ now is given by
    \begin{equation}
        \tilde{\zeta} :=  \frac{\xi}{\alpha_1\,\alpha_2} \, \rho.
    \end{equation}
    We conclude that both models could be derived from the multi-species Boltzmann equation \eqref{BEs_multi_regime}, by considering different closure assumptions and different quantities of interest.
\section{Conclusions}
In this paper we have formally derived the hyperbolic model introduced in \cite{romenski2004compressible} to reproduce the dynamics of two-phase flows, starting from a kinetic model for gas mixtures. We have first investigated the properties of the multi-species Boltzmann equations and we have given a brief description of the behavior of the limit solution. We have then examined the hydrodynamic limit of the kinetic model under the regime of resonant intra-species collisions, aiming to derive different equilibrium velocities for the two species. This choice allows to recover a description for two-phase flows at the macroscopic level. We have kept an isentropic closure assumption on the pressure term, in order to recover the original model. This assumption could be weakened in future works, leading to different macroscopic equations for multi-phase flows, see for instance \cite{thomann2023}.

Several new directions could be pursued. The idea is to develop an asymptotic-preserving numerical scheme for the stiff kinetic model \eqref{BEs_multi_regime}. This will enable numerical simulations, to provide a comparison between the behavior of the kinetic solution and that of the macroscopic equations \eqref{RDT_1}-\eqref{RDT_5} in the limit $\eps \to 0$. Further, we could extend this strategy to more complex kinetic models, taking into account Boltzmann equations with different type of collisions, such as mixtures of gases undergoing inelastic collisions \cite{reytenna2024}. 

\section*{Acknowledgement}
GP and TT would like to thank Michael Dumbser for the fruitful discussion on this topic. TT would also like to thank Andrea Bondesan for his help.\\
The authors received funding from the European Union's Horizon Europe research and innovation program under the Marie Skłodowska-Curie Doctoral Network DataHyking (Grant No. 101072546). TR and TT are also supported by the French government, through the UniCA$_{JEDI}$ Investments in the Future project managed by the National Research Agency (ANR) with the reference number ANR-15-IDEX-01.
46). GP is also supported by MUR (Ministry of University and Research) under the PRIN-2022 project on “High order structure-preserving semi-implicit schemes for hyperbolic equations” (number 2022JH87B4).\\
GP and TT are members of the INdAM Research National Group of Scientific Computing (INdAM-GNCS).\\

\subsection*{Data Availability} Data sharing not applicable to this article as no datasets were generated or analysed during the current study.

\subsection*{Conflict of interest} No conflict of interest is extant in the present work.

\appendix

\section{Computation of the momentum exchange rate for hard-spheres molecules}
	\label{Subsection_Computation_Exchange}
	We compute here the momentum exchange rate $\mathbf{R}_{p, q}$, describing the change in momentum for the species $p$ due to the interactions with the other species $q$. We partially follow \cite{burgers1969} and the more recent paper \cite{Hunana_2022}. We assume a collision kernel $\B$ independent from the species. Since we are interested in an approximation of the exchange rate in the limit for $\eps \to 0$, it is reasonable to consider $f_p$ and $f_q$ as two Maxwellian distribution functions. In particular we can write the exchange rate in the following form
	\begin{equation}
		\mathbf{R}_{p,q} = \int_{\R^d \times \R^d \times \mathbb{S}^{d-1}} m_p\, m_q \mathcal{M}_p(v) \mathcal{M}_q(v_*) (v-v_*) \B(|v-v_*|,\cos\theta) d\sigma dv\,dv_* .
	\end{equation}
	Let us consider the product of two local Maxwellian
	\begin{equation}
		\label{product_maxwellians}
		\mathcal{M}_p(v) \, \mathcal{M}_q(v_*)= n_p \left(\frac{m_p}{2\pi T}\right)^{\frac{d}{2}} \exp \left(-\frac{m_p}{2\,T} |\bar{u}_p-v|^2 \right)\,n_q \left(\frac{m_q}{2\pi T}\right)^{\frac{d}{2}} \exp \left(-\frac{m_q}{2\,T} |\bar{u}_q-v_*|^2 \right).
	\end{equation}
	We introduce the following relative center of mass with respect to the velocities at the equilibrium. We can define 
	\begin{equation}
		\mathcal{W} = \frac{m_p\, \bar{u}_p + m_q\, \bar{u}_q }{m_p+m_q} - \mathcal{V}, 
	\end{equation}
	where $\mathcal{V}$ is the center of mass defined in \eqref{center_mass}. Now, we can rewrite 
	\begin{align*}
		&\bar{u}_p - v = \mathcal{W} + \frac{m_q}{m_p+m_q}[(\bar{u}_p-v)-(\bar{u}_q-v_*)],\\
		&\bar{u}_q - v_* = \mathcal{W} + \frac{m_p}{m_p+m_q}[(\bar{u}_q-v_*)-(\bar{u}_p-v)].
	\end{align*}
	Using this new variable, we can rewrite the product \eqref{product_maxwellians} as
	\begin{equation}
		\mathcal{M}_p(v) \, \mathcal{M}_q(v_*) = K_{pq} \exp \left(  -\frac{|\mathcal{W}|^2}{\mu_{pq}} - \frac{|(v-v_*) - (\bar{u}_p - \bar{u}_q)|^2}{\gamma_{pq}}\right),
	\end{equation}
	where 
	\begin{equation}
		K_{pq} = n_p\,n_q \frac{(m_p\,m_q)^{ \frac{d}{2}}}{(2\pi T)^d}, \quad \mu_{pq} = \frac{2T}{(m_p + m_q)}, \quad \gamma_{pq} = \frac{2(m_p+m_q) \, T}{m_p\,m_q}.
	\end{equation}
	Now we can perform the integral with respect the new variables $\mathcal{W}$ and $w:=v-v_*$, observing that the Jacobian of the transformation is unitary, which implies $dv\,dv_* = dw\,d\mathcal{W}$. Hence, integrating firstly over $d\mathcal{W}$:
	\begin{equation}
		\mathbf{R}_{p,q} = \pi^{\frac{d}{2}} \mu_{pq}^d K_{pq} m_p\,m_q \int_{\R^d \times \mathbb{S}^{d-1}}  \B(|w|,\cos \theta) \,w \, \exp \left(-\frac{|w - (\bar{u}_p - \bar{u}_q)|^2}{\gamma_{pq}} \right) d\sigma \, dw
	\end{equation}
	We now restrict our analysis to the hard-spheres case for $d=3$, where the interaction kernel is 
	\begin{equation}
		\B(|w|, \cos \theta) = |w|,
	\end{equation}
	from which we obtain
	\begin{equation}
		\mathbf{R}_{p,q} = 2 \pi^{\frac{5}{2}} \mu_{pq}^3 K_{pq} m_p\,m_q \int_{\R^3} 	|w| w \exp \left(- \frac{|w - (\bar{u}_p - \bar{u}_q)|^2}{\gamma_{pq}} \right) \,dw.
	\end{equation}
	Let us rewrite $w$ in spherical coordinates
	\begin{equation}
		w=|w|(\sin\varphi \cos \psi, \sin \varphi \sin \psi, \cos \varphi).
	\end{equation}
	Then,
	\begin{align*}
		I_{p,q}:=\int_{\R^3} |w| w &\exp \left(- \frac{|w - (\bar{u}_p - \bar{u}_q)|^2}{\gamma_{pq}} \right) \,dw\\ =&\int_0^\pi \int_0^{2\pi} \int_{0}^\infty |w|^3 \hat{w}\exp \left(- \frac{|w - (\bar{u}_p - \bar{u}_q)|^2}{\gamma_{pq}} \right) \sin \varphi \,d\varphi \,d\psi \,d|w| \\
		=& 2\pi \hat{e}_3 \int_0^{\pi} \int_0^\infty |w|^4 \exp \left(- \frac{|w - (\bar{u}_p - \bar{u}_q)|^2}{\gamma_{pq}} \right) \cos \varphi \sin \varphi \,d\varphi \,d|w|,
	\end{align*}
	where $\hat{w} = w/|w|$ and $\hat{e}_3$ could be chosen in the direction of $\bar{u}_p-\bar{u}_q$.\\
	Introducing the first change of variable
	\begin{equation}
		\alpha = (\bar{u}_p-\bar{u}_q)\cos \varphi, \qquad \beta = |w|, \qquad \nu = \bar{u}_p-\bar{u}_q,
	\end{equation}
	the integral becomes
	\begin{equation}
		I_{p,q}=\frac{2 \pi}{|\nu|^3} \nu \exp \left(-\frac{\nu^2}{\gamma_{pq}} \right) \int_{-\nu}^\nu \int_0^\infty \beta^4 \alpha \exp \left(-\frac{\beta^2}{\gamma_{pq}} -\frac{2\beta\alpha}{\gamma_{pq}} \right) d\alpha d\beta.
	\end{equation}
	Now we introduce another change of variable
	\begin{equation}
		\eta=\beta+\alpha,
	\end{equation}
	in order to rewrite $I_{p,q}$ as
	\begin{equation*}
		I_{p,q} = \frac{2 \pi}{|\nu|^3} \nu \exp \left(-\frac{\nu^2}{\gamma_{pq}} \right) \int_{-\nu}^\nu \int_0^\infty \alpha (\eta-\alpha)^4 \exp \left(-\frac{\eta^2}{\gamma_{pq}} +\frac{\alpha^2}{\gamma_{pq}} \right) d\alpha d\eta.
	\end{equation*}
	Let us consider the rescaled variables
	\begin{equation*}
		\eta = \frac{\eta}{\sqrt{\gamma_{pq}}}, \qquad \alpha = \frac{\alpha}{\sqrt{\gamma_{pq}}}. 
	\end{equation*}
	Now, avoiding the label in $\gamma_{pq}$, we have
	\begin{equation*}
		I_{p,q} = \frac{2 \pi}{|\nu|^3} \gamma^{11/2} \, \nu \exp \left(-\frac{\nu^2}{\gamma} \right) \int_{-\frac{\nu}{\gamma}}^{\frac{\nu}{\gamma}} \int_0^\infty \alpha (\eta-\alpha)^4 \exp \left(-\eta^2 +\alpha^2 \right) d\alpha \, d\eta.
	\end{equation*}
	We first perform the integration in $d\eta$,
	\begin{align*}
		\int_\alpha^\infty &(\eta-\alpha)^4 \exp \left(-\eta^2 \right) d\eta = \\
		=&\frac{1}{8} \sqrt{\pi} (3+ 12 \alpha^2 + 4 \alpha^4)\, \text{erfc}(\alpha) - \frac{1}{4} \alpha \exp \left(-\alpha^2 \right) (5+2\alpha^2),
	\end{align*}
	where erfc is the complementary error function.\\
	The final expression for $I_{p,q}$ is
	\begin{align*}
	\small
		I_{p,q} = &\frac{ \pi}{4|\nu|^3} \gamma^{11/2} \, \nu e^{\left(-\frac{\nu^2}{\gamma} \right)} \left(-4\left(\frac{\nu}{\gamma} \right)^3+ 2\left(\frac{\nu}{\gamma} \right)- \right. \\ & \left. \sqrt{\pi} \left(4\left(\frac{\nu}{\gamma} \right)^4+4\left(\frac{\nu}{\gamma} \right)^2-1 \right)e^{\left(\frac{\nu}{\gamma}\right)^2}\text{erf}\left(\frac{\nu}{\gamma} \right)\right)\\ \small
		= & \pi \nu \left[e^{\left(-\frac{\nu^2}{\gamma} \right)} \left(\gamma^{5/2} + \frac{\gamma^{9/2}}{2\nu^2}\right) - \sqrt{\pi} \left( \nu \gamma^{3/2} + \frac{\gamma^{7/2}}{\nu} - \frac{\gamma^{11/2}}{4\nu^3} \right) \text{erf}\left(\frac{\nu}{\gamma} \right)\right].
	\end{align*}
	We can conclude that
	\begin{equation*}
		\mathbf{R}_{p,q} =   \rho_p \, \rho_q (\bar{u}_p - \bar{u}_q) \Psi_{p,q},
	\end{equation*}
	where $\displaystyle \Psi_{p,q} = \frac{(m_p\,m_q)^{\frac{d}{2}}}{\sqrt{\pi}(m_p+m_q)} \frac{I_{p,q}}{(\bar{u}_p-\bar{u}_q)}$.
\bibliography{References}
\bibliographystyle{acm}

\end{document}